\documentclass{amsart}

\usepackage[normalem]{ulem}
\usepackage{mathtools}
\usepackage{mathrsfs}  
\usepackage[arrow,curve,matrix]{xy}

\usepackage{comment}
\usepackage[utf8]{inputenc}

\usepackage{amsmath,amsfonts,amsthm,mathrsfs}
\usepackage{amssymb}
\usepackage{mathrsfs}

\usepackage[bookmarks]{hyperref}
\usepackage[usenames,dvipsnames]{xcolor}
\hypersetup{colorlinks=true,citecolor=NavyBlue,linkcolor=Maroon,urlcolor=Orange}

\usepackage[alphabetic,initials]{amsrefs}

\usepackage{enumitem}
\usepackage{chngcntr}

\usepackage{tikz}

\usepackage{tikz-cd}
\tikzset{commutative diagrams/arrow style=math font}

\usetikzlibrary{matrix,arrows}
\newlength{\myarrowsize} 

\pgfarrowsdeclare{cmto}{cmto}{
	\pgfsetdash{}{0pt} 
	\pgfsetbeveljoin 
	\pgfsetroundcap 
	\setlength{\myarrowsize}{0.6pt}
	\addtolength{\myarrowsize}{.5\pgflinewidth}
	\pgfarrowsleftextend{-4\myarrowsize-.5\pgflinewidth} 
	\pgfarrowsrightextend{.8\pgflinewidth}
}{
	\setlength{\myarrowsize}{0.6pt} 
  	\addtolength{\myarrowsize}{.5\pgflinewidth}  
	\pgfsetlinewidth{0.5\pgflinewidth}
	\pgfsetroundjoin
	\pgfpathmoveto{\pgfpoint{1.5\pgflinewidth}{0}}
	\pgfpatharc{-109}{-170}{4\myarrowsize}
	\pgfpatharc{10}{189}{0.58\pgflinewidth and 0.2\pgflinewidth}
	\pgfpatharc{-170}{-115}{4\myarrowsize+\pgflinewidth}
	\pgfpathclose
	\pgfusepathqfillstroke
	\pgfpathmoveto{\pgfpoint{1.5\pgflinewidth}{0}}
	\pgfpatharc{109}{170}{4\myarrowsize}
	\pgfpatharc{-10}{-189}{0.58\pgflinewidth and 0.2\pgflinewidth}
	\pgfpatharc{170}{115}{4\myarrowsize+\pgflinewidth}
	\pgfpathclose
	\pgfusepathqfillstroke
	\pgfsetlinewidth{2\pgflinewidth}
}

\pgfarrowsdeclare{cmonto}{cmonto}{
	\pgfsetdash{}{0pt} 
	\pgfsetbeveljoin 
	\pgfsetroundcap 
	\setlength{\myarrowsize}{0.6pt}
	\addtolength{\myarrowsize}{.5\pgflinewidth}
	\pgfarrowsleftextend{-4\myarrowsize-.5\pgflinewidth} 
	\pgfarrowsrightextend{.8\pgflinewidth}
}{
	\setlength{\myarrowsize}{0.6pt} 
  	\addtolength{\myarrowsize}{.5\pgflinewidth}  
	\pgfsetlinewidth{0.5\pgflinewidth}
	\pgfsetroundjoin
	\pgfpathmoveto{\pgfpoint{1.5\pgflinewidth}{0}}
	\pgfpatharc{-109}{-170}{4\myarrowsize}
	\pgfpatharc{10}{189}{0.58\pgflinewidth and 0.2\pgflinewidth}
	\pgfpatharc{-170}{-115}{4\myarrowsize+\pgflinewidth}
	\pgfpathclose
	\pgfusepathqfillstroke
	\pgfpathmoveto{\pgfpoint{1.5\pgflinewidth}{0}}
	\pgfpatharc{109}{170}{4\myarrowsize}
	\pgfpatharc{-10}{-189}{0.58\pgflinewidth and 0.2\pgflinewidth}
	\pgfpatharc{170}{115}{4\myarrowsize+\pgflinewidth}
	\pgfpathclose
	\pgfusepathqfillstroke
	\pgfpathmoveto{\pgfpoint{1.5\pgflinewidth-0.3em}{0}}
	\pgfpatharc{-109}{-170}{4\myarrowsize}
	\pgfpatharc{10}{189}{0.58\pgflinewidth and 0.2\pgflinewidth}
	\pgfpatharc{-170}{-115}{4\myarrowsize+\pgflinewidth}
	\pgfpathclose
	\pgfusepathqfillstroke
	\pgfpathmoveto{\pgfpoint{1.5\pgflinewidth-0.3em}{0}}
	\pgfpatharc{109}{170}{4\myarrowsize}
	\pgfpatharc{-10}{-189}{0.58\pgflinewidth and 0.2\pgflinewidth}
	\pgfpatharc{170}{115}{4\myarrowsize+\pgflinewidth}
	\pgfpathclose
	\pgfusepathqfillstroke
	\pgfsetlinewidth{2\pgflinewidth}
}

\pgfarrowsdeclare{cmhook}{cmhook}{
	\pgfsetdash{}{0pt} 
	\pgfsetbeveljoin 
	\pgfsetroundcap 
	\setlength{\myarrowsize}{0.6pt}
	\addtolength{\myarrowsize}{.5\pgflinewidth}
	\pgfarrowsleftextend{-4\myarrowsize-.5\pgflinewidth} 
	\pgfarrowsrightextend{.8\pgflinewidth}
}{
	\setlength{\myarrowsize}{0.6pt} 
  	\addtolength{\myarrowsize}{.5\pgflinewidth}  
 	\pgfsetdash{}{0pt}
	\pgfsetroundcap
	\pgfpathmoveto{\pgfqpoint{0pt}{-4.667\pgflinewidth}}
	\pgfpathcurveto
    {\pgfqpoint{4\pgflinewidth}{-4.667\pgflinewidth}}
    {\pgfqpoint{4\pgflinewidth}{0pt}}
    {\pgfpointorigin}
	\pgfusepathqstroke
}

\newenvironment{diagram*}[2]{%
\[%
\begin{tikzpicture}[>=cmto,baseline=(current bounding box.center),%
	to/.style={->,font=\scriptsize,cap=round},%
	into/.style={cmhook->,font=\scriptsize,cap=round},%
	onto/.style={-cmonto,font=\scriptsize,cap=round},%
	math/.style={matrix of math nodes, row sep=#2, column sep=#1,%
		text height=1.5ex, text depth=0.25ex}]%
}{%
\end{tikzpicture}%
\]%
\ignorespacesafterend%
}

\DeclareMathOperator{\rk}{rk}

\DeclareMathOperator{\Div}{Div}

\DeclareMathOperator{\id}{id}

\DeclareMathOperator{\Supp}{Supp}
\DeclareMathOperator{\codim}{codim}

\DeclareMathOperator{\Var}{Var}

\newcommand{\sA}{\scr{A}}
\newcommand{\sB}{\scr{B}}

\newcommand{\sE}{\scr{E}}
\newcommand{\sL}{\scr{L}}
\newcommand{\sF}{\scr{F}}
\newcommand{\sG}{\scr{G}}

\newcommand{\sM}{\scr{M}}
\newcommand{\sN}{\scr{N}}
\newcommand{\sO}{\scr{O}}
\newcommand{\sP}{\scr{P}}
\newcommand{\sQ}{\scr{Q}}
\newcommand{\sT}{\scr{T}}
\newcommand{\sW}{\scr{W}}

\newcommand{\theoremref}[1]{\hyperref[#1]{Theorem~\ref*{#1}}}
\newcommand{\lemmaref}[1]{\hyperref[#1]{Lemma~\ref*{#1}}}
\newcommand{\definitionref}[1]{\hyperref[#1]{Definition~\ref*{#1}}}
\newcommand{\propositionref}[1]{\hyperref[#1]{Proposition~\ref*{#1}}}
\newcommand{\conjectureref}[1]{\hyperref[#1]{Conjecture~\ref*{#1}}}
\newcommand{\corollaryref}[1]{\hyperref[#1]{Corollary~\ref*{#1}}}
\newcommand{\exampleref}[1]{\hyperref[#1]{Example~\ref*{#1}}}

\makeatletter
\let\old@caption\caption
\renewcommand*{\caption}[1]{%
	\setcounter{figure}{\value{equation}}%
	\stepcounter{equation}%
	\old@caption{#1}\relax%
}
\makeatother

\newcounter{intro}

\newtheorem{intro-conjecture}[intro]{Conjecture}
\newtheorem{intro-corollary}[intro]{Corollary}
\newtheorem{intro-theorem}[intro]{Theorem}

\newcommand{\parref}[1]{\hyperref[#1]{\S\ref*{#1}}}

\makeatletter
\newcommand*\if@single[3]{%
  \setbox0\hbox{${\mathaccent"0362{#1}}^H$}%
  \setbox2\hbox{${\mathaccent"0362{\kern0pt#1}}^H$}%
  \ifdim\ht0=\ht2 #3\else #2\fi
  }
\newcommand*\rel@kern[1]{\kern#1\dimexpr\macc@kerna}
\newcommand*\widebar[1]{\@ifnextchar^{{\wide@bar{#1}{0}}}{\wide@bar{#1}{1}}}
\newcommand*\wide@bar[2]{\if@single{#1}{\wide@bar@{#1}{#2}{1}}{\wide@bar@{#1}{#2}{2}}}
\newcommand*\wide@bar@[3]{%
  \begingroup
  \def\mathaccent##1##2{%
    \if#32 \let\macc@nucleus\first@char \fi
    \setbox\z@\hbox{$\macc@style{\macc@nucleus}_{}$}%
    \setbox\tw@\hbox{$\macc@style{\macc@nucleus}{}_{}$}%
    \dimen@\wd\tw@
    \advance\dimen@-\wd\z@
    \divide\dimen@ 3
    \@tempdima\wd\tw@
    \advance\@tempdima-\scriptspace
    \divide\@tempdima 10
    \advance\dimen@-\@tempdima
    \ifdim\dimen@>\z@ \dimen@0pt\fi
    \rel@kern{0.6}\kern-\dimen@
    \if#31
      \overline{\rel@kern{-0.6}\kern\dimen@\macc@nucleus\rel@kern{0.4}\kern\dimen@}%
      \advance\dimen@0.4\dimexpr\macc@kerna
      \let\final@kern#2%
      \ifdim\dimen@<\z@ \let\final@kern1\fi
      \if\final@kern1 \kern-\dimen@\fi
    \else
      \overline{\rel@kern{-0.6}\kern\dimen@#1}%
    \fi
  }%
  \macc@depth\@ne
  \let\math@bgroup\@empty \let\math@egroup\macc@set@skewchar
  \mathsurround\z@ \frozen@everymath{\mathgroup\macc@group\relax}%
  \macc@set@skewchar\relax
  \let\mathaccentV\macc@nested@a
  \if#31
    \macc@nested@a\relax111{#1}%
  \else
    \def\gobble@till@marker##1\endmarker{}%
    \futurelet\first@char\gobble@till@marker#1\endmarker
    \ifcat\noexpand\first@char A\else
      \def\first@char{}%
    \fi
    \macc@nested@a\relax111{\first@char}%
  \fi
  \endgroup
}
\makeatother

\DeclareFontFamily{OMS}{rsfs}{\skewchar\font'60}
\DeclareFontShape{OMS}{rsfs}{m}{n}{<-5>rsfs5 <5-7>rsfs7 <7->rsfs10 }{}
\DeclareSymbolFont{rsfs}{OMS}{rsfs}{m}{n}
\DeclareSymbolFontAlphabet{\scr}{rsfs}

\numberwithin{figure}{section}

\DeclareMathOperator{\disc}{disc}
\DeclareMathOperator{\NE}{NE}
\DeclareMathOperator{\Gal}{Gal}
\DeclareMathOperator{\I}{I}
\DeclareMathOperator{\R}{R}
\DeclareMathOperator{\dR}{\mathbf{R}}

\newcommand{\bC}{\mathbb{C}}
\newcommand{\bN}{\mathbb{N}}
\newcommand{\bQ}{\mathbb{Q}}

\newcommand{\hooklongrightarrow}{\lhook\joinrel\longrightarrow}

\theoremstyle{plain}
\newtheorem{theorem}{Theorem}[section]

\newtheorem{proposition}[theorem]{Proposition}

\newtheorem{lemma}[theorem]{Lemma}
\newtheorem{conjecture}[theorem]{Conjecture}
\newtheorem{question}[theorem]{Question}
\theoremstyle{definition}
\newtheorem{definition}[theorem]{Definition}

\theoremstyle{remark}
\newtheorem{remark}[theorem]{Remark}
\newtheorem{notation}[theorem]{Notation}

\newtheorem{set-up}[theorem]{Set-up}
\newtheorem{claim}[theorem]{Claim}

\setlist[enumerate]{label=(\thetheorem.\arabic*), before={\setcounter{enumi}{\value{equation}}}, after={\setcounter{equation}{\value{enumi}}}}

\numberwithin{equation}{theorem}

\begin{document}

\title[Kodaira dimension of base spaces]{On the Kodaira dimension of base spaces of families of manifolds}

\author{Behrouz Taji} \address{Behrouz Taji, School of Mathematics and Statistics
  F07, The University of Sydney, NSW 2006 Australia}
\email{\href{mailto:behrouz.taji@sydney.edu.au}{behrouz.taji@sydney.edu.au}}
\urladdr{\href{http://www.maths.usyd.edu.au/u/behrouzt/}
  {http://www.maths.usyd.edu.au/u/behrouzt/}}

\keywords{Families of manifolds, minimal models, Kodaira dimension, variation of Hodge structures,
 moduli of polarized varieties.}

\subjclass[2010]{14D06, 14D23, 14E05, 14E30, 14D07.}


\setlength{\parskip}{0.19\baselineskip}


\begin{abstract}
We prove that the variation in a smooth projective family of varieties  
admitting a good minimal model forms a lower bound for the Kodaira dimension of the 
base, if the dimension of the base is at most five and its Kodaira dimension is non-negative. 
This gives an affirmative answer to 
the conjecture of Kebekus and Kov\'acs
for base spaces of dimension at most five.
\end{abstract}

\maketitle


\section{Introduction and main results}
\label{sect:Section1-Introduction}

\subsection{Introduction}
A conjecture of Viehweg~\cite{Viehweg01}, generalizing Shafarevich Conjecture
for family of curves~\cite{Shaf63}, predicted that
if fibers of a smooth projective family $f_U:U\to V$ are
canonically polarized, then $\kappa(V) = \dim(V)$, assuming that $f_U$ has  
\emph{maximal variation}. The problem was subsequently generalized to the case of 
fibers with good minimal models, cf.~\cite{Vie-Zuo01} and \cite{PS15}.
This conjecture is usually referred to as \emph{Viehweg Hyperbolicity Conjecture}
and has been recently settled in full generality.

Generalizing Viehweg's conjecture, 
in their groundbreaking series of papers, Kebekus and Kov\'acs  
predicted that variation in the smooth family 
$f_U$, which we denote by $\Var(f_U)$, should be closely connected to $\kappa(V)$, even when it is \emph{not} maximal.

\begin{conjecture}[\protect{Kebekus-Kov\'acs Hyperbolicity Conjecture.~I, cf.~\cite[Conj.~1.6]{KK08}}]\label{conj:kk0}
Let $f_U: U\to V$ be a smooth projective family whose general fiber
admits a good minimal model \footnote{The original conjecture of Kebekus and Kov\'acs was formulated
for the case of canonically polarized fibers.}. Then, either
\begin{enumerate}
\item $\kappa(V) = -\infty$ and $\Var(f_U) < \dim V$, or 
\item $\kappa(V)\geq 0$ and $\Var(f_U)\leq \kappa(V)$.
\end{enumerate}
\end{conjecture}

Throughout this paper $U$ and $V$ are smooth quasi-projective varieties
and $f_U$ has connected fibers.

Once the family $f_U$ arises from a moduli functor with an algebraic coarse moduli scheme,
then an even stronger version of Conjecture~\ref{conj:kk0}, that is due to Campana,
can be verified (see Section~\ref{sect:Section5-Future}). However, the main 
focus of this paper is to study Conjecture~\ref{conj:kk0}, when the family 
$f_U$ is not associated with a well-behaved moduli functor, even after running a 
relative minimal model program. 

The following theorem is the main result of this paper.

\begin{theorem}\label{thm:main}
Conjecture~\ref{conj:kk0} holds when $\dim(V)\leq 5$.
\end{theorem}

When dimension of the base and fibers are equal to one, Viehweg's 
hyperbolicity conjecture (or equivalently Conjecture~\ref{conj:kk0}) was proved by Parshin
\cite{Parshin68}, in the compact case,
and in general by Arakelov \cite{Arakelov71}. For higher dimensional fibers and assuming 
that $\dim(V)=1$, this conjecture was confirmed by Kov\'acs~\cite{Kovacs00a}, in the canonically polarized case
(see also~\cite{Kovacs02}), 
and by Viehweg and Zuo~\cite{Vie-Zuo01} in general.
Over Abelian varieties Viehweg's conjecture was solved by Kov\'acs~\cite{Kovacs97c}.
When $\dim(V)=2$ or $3$, it was resolved by Kebekus and Kov\'acs, cf.~\cite{KK08} and \cite{KK10}.
In the compact case it was settled by Patakfalvi~\cite{MR2871152}.
Viehweg's conjecture was finally solved in complete generality by the fundamental work of 
Campana and P\u{a}un \cite{CP16} and more recently by Popa and 
Schnell \cite{PS15}. 
For the more analytic counterparts of these results please see~\cite{Vie-Zuo03a},
\cite{Sch12}, \cite{TY15}, \cite{BPW}, \cite{TY16}, \cite{PTW} and \cite{Deng}.

By using the solution of Viehweg's hyperbolicity conjecture, one can reformulate  
Conjecture~\ref{conj:kk0} as follows.

\begin{conjecture}[Kebekus-Kov\'acs Hyperbolicity Conjecture.~II]\label{conj:kk}
Let $f_U : U \to V$ be a smooth family of projective varieties and $(X,D)$ a smooth 
compactification of $V$ with $V \cong X\smallsetminus D$. Assume that 
the general fiber of $f_U$ has a good minimal model.
Then, the inequality 

$$
\Var(f_U) \leq \kappa(X, D)
$$
holds, if $\kappa(X,D)\geq  0$.
\end{conjecture}

When $f_U$ is canonically polarized, Conjecture~\ref{conj:kk} was confirmed 
by Kebekus and Kov\'acs in \cite{KK08}, assuming that $\dim(X)=2$, in~\cite{KK10}, 
if $\dim(X)=3$, and as a consequence of \cite{taji16} in full generality. The latter result 
establishes an independent, but closely related, conjecture 
of Campana; the so-called \emph{Isotriviality Conjecture}.

Campana's conjecture (Conjecture~\ref{conj:iso}) predicts that 
once the fibers of $f_U$ are canonically polarized (or more generally have 
semi-ample canonical bundle, assuming that the family is polarized with a fixed Hilbert polynomial
in the sense of Viehweg \cite[Thm.~1.13]{Viehweg95}), 
then $f_U$ is isotrivial, 
if $V$ is \emph{special}.
We refer to Section~\ref{sect:Section5-Future} for more details on the notion of
special varieties
and various other particular cases where the Isotriviality Conjecture and consequently 
Conjecture~\ref{conj:kk} can be confirmed, providing further evidence for the
importance of the theory of special varieties. 
Briefly, the key point is that
whenever one can prove that smooth projective families of fixed Kodaira dimension and admitting 
good minimal models are isotrivial over special varieties, then 
the inequality in Conjecture~\ref{conj:kk} is valid for such families.

\subsection{Brief review of the strategy of the proof}
A key tool in proving Conjecture~\ref{conj:kk}, in the canonically polarized case,
is the celebrated result of Viehweg and Zuo \cite[Thm.~1.4.(i)]{Vie-Zuo01} on the existence of 
an invertible subsheaf $\sL\subseteq (\Omega^1_X(\log D))^{\otimes i}$, for some $i\in \bN$, whose Kodaira dimension verifies
the inequality: 

\begin{equation}\label{eq:VZ-ineq}
 \kappa(X, \sL) \geq \Var(f_U).
\end{equation}
The sheaf $\sL$ is usually referred to as a \emph{Viehweg-Zuo subsheaf}.
In general, once the canonically polarized condition is dropped, in the 
absence of a well-behaved moduli functor associated to the family, 
the approach of \cite{Vie-Zuo01} cannot be directly applied.
Nevertheless, we show that one can still construct a subsheaf of $(\Omega^1_X(\log D))^{\otimes i}$
arising from the variation in $f_U$, as long as the general fiber of 
$f_U$ has a good minimal model. But in this more general context 
the sheaf $\sL$ injects into $(\Omega^1_X(\log D))^{\otimes i}$ only after it is twisted by some pseudo-effective 
line bundle $\sB$. As such we can no longer 
guarantee that this (twisted) subsheaf verifies the inequality (\ref{eq:VZ-ineq}).

\medskip

\begin{theorem}[Existence of pseudo-effective Viehweg-Zuo subsheaves]\label{thm:VZ}
Let $f_U: U\to V$ be a smooth non-isotrivial family whose general fiber admits 
a good minimal model. Let $(X, D)$ a smooth compactification of $V$.
There exist a positive integer $i$, a line bundle $\sL$ on $X$, 
with $\kappa(X, \sL) \geq \Var(f_U)$, and a pseudo-effective 
line bundle $\sB$ with an inclusion 

\begin{equation}\label{inclusion}
\sL \otimes \sB  \subseteq  (\Omega^1_X(\log D))^{\otimes i}.
\end{equation}

\end{theorem}

To mark the difference between the two cases 
we refer to these newly constructed sheaves as
\emph{pseudo-effective Viehweg-Zuo subsheaves}.

\medskip

\subsubsection{Hodge theoretic constructions.} Assuming that $\Var(f_U) = \dim(X)$, 
the line bundle in Theorem~\ref{thm:VZ} is then big and we have
$\kappa(X, \sL\otimes \sB) = \dim(X)$. Thus, in this case, Theorem~\ref{thm:VZ}
coincides with the result of Viehweg-Zuo \cite[Thm.~1.4.(i)]{VZ02}, when 
fibers are canonically polarized, and the theorem of Popa-Schnell \cite[Thm.~B]{PS15}
in the case of more general fibers. 
Although the general strategy in this paper 
is similar to~\cite{VZ02}, there are some conceptual differences,
which I believe make the proof (including in the maximal variation case) more 
accessible. Let us now sketch the key points in this construction.

Let $f: Y\to X$ be a smooth compactification of $f_U$ and for simplicity 
let us assume that it is semistable in codimension one. As in the case
of~\cite{VZ02} and~\cite{PS15}, the key first step, is to use the fact that 
the line bundle $\det(f_* \omega_{Y/X}^{a})^{**}$ is big, for some $a\in \bN$,
cf.~\cite{Kawamata85},
to construct a smooth projective variety $Z$ and  
a generically finite map $\pi: Z\to Y$ arising from the global sections of $\big(\omega_{Y/X} \otimes f^* \sL^{-1}\big)^m$, 
with $\sL$ being a big line bundle and $m$ sufficiently large.
We note that this construction is only valid 
after the initial family is replaced by a   
large fiber product $\underbrace{Y \times_X \ldots \times_X Y}_{r}$, for sufficiently large $r$.

At this point we use the geometric data above, which is a consequence of our assumptions 
on the variation in the family, to construct a subsystem 

$$
(\sG, \theta) \subseteq (\sE, \theta),
$$
where $(\sE = \bigoplus \sE_{i}, \theta)$ is the graded Higgs bundle underlying the variation of Hodge 
structures of the family $g =  f\circ \pi: Z\to X$.
The system $(\sG, \theta)$ verifies the property
that $\theta(\sG)  \subset \Omega^1_X(\log D) \otimes \sG$. Furthermore, the first 
graded piece $\sG_0$ of $\sG$ admits a non-trivial morphism $\sL \longrightarrow \sG_0$.
Existence of such $(\sG, \theta)$ then leads to the inclusion (\ref{inclusion})
in Theorem~\ref{thm:VZ}, thanks to negativity of Hodge metrics along the kernel of 
$\theta|_{\sE_{i}}$ (see Section \ref{sect:Section3-VZ} for the details and references).

The construction of $(\sG, \theta)$ is where the two approaches 
of~\cite{VZ02} and~\cite{PS15} differ. 
While the former relies on strong positivity results for direct image sheaves
for families of manifolds with semiample canonical bundle, the latter  
resorts to Hodge modules and the decomposition theorem.
To a large extent, the approach in this paper follows the original general strategy of~\cite{VZ02}
but is purely based on natural functorial properties of 
de Rham complexes.
This is the content of Sections~\ref{sect:Section2-Hodge} and~\ref{sect:Section3-VZ}, 
and in the more general 
setting of non-isotrivial families (instead of just the ones with maximal variation).

It is worth pointing out that once there is an algebraic moduli functor (for example
in the canonically polarized case), then the line bundle $\sB$ in Theorem~\ref{thm:VZ}
is trivial. The reason that in general $\sB$ cannot be assumed to be trivial is that
without a reasonable functor 
associated to the family, the existence of Viehweg-Zuo subsheaves 
can no longer be extracted from their construction at the level of moduli stacks,
where the variation is maximal. We refer the reader to Section~\ref{sect:Section5-Future}
for more details and a brief review of the cases where this difficulty can be overcome.

\subsubsection{Birational methods and vanishing results.}
Once Theorem~\ref{thm:VZ} is established, the only issues that remains is to trace a connection 
between $\kappa(\sL)$ and $\kappa(X,D)$.
In the maximal variation case, Viehweg-Zuo subsheaves 
are guaranteed to be big (as soon as the general fiber has 
a good minimal model). In this case a key result of Campana and P\u{a}un 
then implies that $\kappa(X,D) = \dim(X)$, cf.~\cite[Thm~7.11]{CP16}.
But when $\Var(f_U)<\dim(X)$, as $\sL$ is not big, this strategy can no longer 
be applied. Instead, our proof of Theorem~\ref{thm:main} relies on the following vanishing 
result.

\medskip

\begin{theorem}[Vanishing for twisted logarithmic pluri-differential forms]
\label{thm:vanishing}
Let $(X,D)$ be a pair consisting of a smooth projective variety $X$ of dimension
at most equal to $5$ and a reduced effective 
divisor $D$ with simple normal crossing support. 
If $\kappa(X, D)\geq 0$, 
then, for any $i\in \bN$, the equality 

\begin{equation*}\label{eq:CamVanishing}
H^0 \big(X, (\Omega^1_X(\log D))^{\otimes i} \otimes  (\sL\otimes \sB)^{-1} \big) = 0
\end{equation*}
holds, assuming that $\sB$ is pseudo-effective and $\kappa(\sL) >  \kappa(X,D)$.
\end{theorem}

Theorem~\ref{thm:main} is now an immediate consequence of Theorems~\ref{thm:vanishing}
and~\ref{thm:VZ}.

In the presence of a flat K\"ahler-Einstein metric, or assuming that 
the main conjectures of the minimal model program hold,
various analogues of Theorem~\ref{thm:vanishing} can be verified. 
When $D = 0$ and $c_1(X)=0$, and $\sB\cong \sO_X$,  
Theorem~\ref{thm:vanishing} follows, in all dimensions, directly from 
Yau's solution to Calabi's conjecture \cite{MR0451180} and the Bochner formula. 
When $\sB\cong \sO_X$ and $D=0$, the vanishing in Theorem~\ref{thm:vanishing} 
was conjectured by Campana in \cite{Ca95} where he proved that it 
holds for an $n$-dimensional variety $X$, if the Abundance Conjecture 
holds in dimension $n$. 
As was shown by Campana, such vanishing results are closely 
related to compactness properties of
the universal cover 
of algebraic varieties.
We refer to \cite{Ca95} for details of this very interesting subject
(see also the book of Koll\'ar~\cite{Kollar95s} and \cite{Kollar93}).
We also invite the reader to consult \cite[Thm.~7.3]{CP16}
where the authors successfully deal with a similar problem with $\kappa$ replaced by another invariant
$\nu$; the latter being the numerical Kodaira dimension.

\subsection{Structure of the paper}
In Section~\ref{sect:Section2-Hodge} we provide the preliminary 
constructions needed for the proof of Theorem~\ref{thm:VZ}. The proof 
of Theorem~\ref{thm:VZ} appears in Section~\ref{sect:Section3-VZ}.
In Section~\ref{sect:Section4-Birational} we prove the vanishing result; 
Theorem~\ref{thm:vanishing}. Section~\ref{sect:Section5-Future} is devoted 
to further results and related problems, including a discussion on the connection 
between Conjecture~\ref{conj:kk} and a conjecture of Campana.

\subsection{Acknowledgements} I would like to thank S\'andor Kov\'acs, Fr\'ed\'eric Campana,
Mihnea Popa, Christian Schnell and Yajnaseni Dutta for their interest and suggestions. I am also grateful to 
the anonymous referee for helpful comments.

\section{Preliminary results and constructions}
\label{sect:Section2-Hodge}
Our aim in this section is to establish two key background results that we will
need in order to construct the pseudo-effective Viehweg-Zuo subsheaves in the proceeding section.

\subsection{Positivity of direct images of relative dualizing sheaves}
In~\cite[Thm~1.1]{Kawamata85} Kawamata shows that, assuming that the general fiber admits 
a good minimal model, for any algebraic fiber space $f: Y\to X$ of smooth
projective varieties $Y$ and $X$, the inequality 

\begin{equation}\label{eq:KAW}
\kappa\Big(X,  \big( \det(f_*\omega^m_{Y/X}) \big)^{**}  \Big) \geq \Var(f),
\end{equation}
holds, for all sufficiently large integers $m\geq 1$\footnote{In~\cite{Kawamata85} this result has been 
established when geometric generic fiber (see p.~5, \emph{loc.cit.}) has a good minimal model.
Using the openness property of good minimal models, one sees that this condition is 
satisfied when the general fiber has a good minimal model.}.

One can use (\ref{eq:KAW}) to extract positivity results for $f_*\omega^m_{Y/X}$.
We refer the reader to~\cite{Vie-Zuo03a}, \cite{VZ02} and \cite{PS15} for the case where $\Var(f) = \dim(X)$ 
(see also~\cite{Kawamata85} and~\cite{Viehweg83}). The key ingredient is the fiber product trick 
of Viehweg, where, given $f: Y\to X$ as above, one considers the $r$-fold fiber product 

$$
Y^r :  =  \underbrace{Y\times_X  Y  \times_X \ldots \times_X Y}_{\text{$r$ times}}.
$$
We denote by $Y^{(r)}$ a desingularization of $Y^{r}$ and the resulting morphism 
by $f^{(r)}:  Y^{(r)}  \to X$.

By standard semistable reduction results (see for example~\cite[Prop.~6.1]{Viehweg83}) 
there is a finite, flat and Galois morphism $\gamma: X_1 \to X$, 
with $G:= \Gal(X_1/ X)$, 
such that 
the induced morphism $f_1: Y_1 \to X_1$ from a $G$-equivariant resolution $Y_1$ of the fiber product 
$Y \times_{X_1} X$ 
is semistable in codimension one:

$$
\xymatrix{
Y_1 \ar[r]  \ar[dr]_{f_1}   &   Y \times_{X_1}  X  \ar[d] \ar[rr]  && Y \ar[d]^f  \\
&                                        X_1  \ar[rr]^{\gamma}_{\text{flat and Galois}}  &&        X.
}
$$
By repeating this construction we find the commutative diagram

$$
\xymatrix{
Y_1^{(r)}  \ar[rr] \ar[d]_{f_1^{(r)}} &&  Y^{(r)} \ar[d]^{f^{(r)}} \\
 X_1  \ar[rr]^{\gamma}    && X. 
}
$$

The next proposition is an extension of the arguments of 
\cite[pp.~708--709]{PS15} to the case where variation is not maximal.

\begin{proposition}\label{prop:positivity}
Let $f: Y\to X$ be a fiber space of smooth projective varieties $Y$ and $X$. If the general 
 fiber admits a good minimal model, then, for every sufficiently large 
$m\geq 1$, there exists $r:= r(m)\in \bN$, a line bundle $\sL$ on $X$, with 
$\kappa(X, \sL)\geq \Var(f)$ and a line bundle $\sL_1$ on $X_1$, where $\kappa(X_1, \sL_1) \geq \Var(f)$, 
verifying the following properties. 

\begin{enumerate}
\item \label{eq:include} Up to isomorphism, and over the locus where
$f^{(r)}_*\big( \omega^m_{Y^{(r)}/ X} \big)$ is reflexive, there are inclusions 
$$
\sL^{m} \subseteq f^{(r)}_*\big( \omega^m_{Y^{(r)}/ X} \big)  \; \; \; and \; \; \;  
   \sL_1^m \subseteq (f_1^{(r)})_* \big( \omega^m_{Y_1^{(r)}/ X_1} \big).
$$

\item \label{item:down1} There is an isomorphism $\sL_1|_{X_1\smallsetminus \disc(f_1)}  \cong  (\gamma^*\sL)|_{X_1\smallsetminus \disc(f_1)}$.
\item \label{item:down2} There is a line bundle $\sA$ on $X$ such that for any sufficiently large 
and divisible integer $d$ we have 
$$
\sL_1^d \cong  \gamma^* \sA.
$$
\end{enumerate}

\end{proposition}

\begin{proof}
Let us define the line bundle $\sL_1$ by 

$$
\sL_1: = \det \big(  (f_1)_*   \omega^{m}_{Y_1/ X_1}   \big).
$$
We can allow ourselves to remove codimension two subsets from $X$. In particular we may 
assume that torsion free sheaves (such as $\sL_1$) are locally free. 

Next, we observe that, as $\Omega^1_{X_1}$ and $\Omega^1_{Y_1}$ are naturally equipped with the structure 
of $G$-sheaves (or linearized sheaves), $\omega^m_{Y_1/X_1}$ is also a $G$-sheaf. It follows that 
$(f_1)_*\omega^m_{Y_1/X_1}$ and its determinant are also $G$-sheaves. 

Now, let $t: = \rk \big( (f_1)_* \omega^m_{Y_1/X_1}   \big)$ and consider the 
natural injection 

$$
\sL_1  \hooklongrightarrow   \bigotimes^{t}  (f_1)_*  \omega^m_{Y_1/X_1}. 
$$
Using the fact that $f_1$ is semistable in codimension one
(and remembering that we are arguing in codimension one), thanks to~\cite[Lem.~3.5]{Viehweg83}, we have

$$
(f_1^{(t)})_*\big( \omega^m_{Y_1^{(t)}/ X_1}  \big)   =   \bigotimes^t  (f_1)_* \big( \omega^m_{Y_1/X_1} \big).
$$
On the other hand, since $\gamma$ is flat, according to~\cite[Lem.~3.2]{Viehweg83}, we have 

$$
(f_1^{(t)})_*  \big( \omega^m_{Y_1^{(t)} / X_1} \big) \subseteq \gamma^* \big(  f^{(t)}_*  \omega^m_{Y^{(t)}/ X}  \big).
$$

Now, let $\sB_1$ be the saturation of the image of $\sL_1$ in $\gamma^* \big(  f^{(t)}_*  \omega^m_{Y^{(t)}/ X}  \big)$.
As the saturation of a $G$-sheaf inside $\gamma^* \big(  f^{(t)}_*  \omega^m_{Y^{(t)}/ X}  \big)$, the line 
bundle $\sB_1$ is again a $G$-sheaf. Therefore, thanks to~\cite[Thm.~4.2.15]{MR2665168}
descends, that is there exists a line bundle $\sL$ on $Y$ such that $\gamma^*\sL \cong \sB_1$.
Moreover as $\gamma$ is finite, it follows that $\kappa(X_1, \sB_1) = \kappa(X, \sL)$.
Note that we also have $\kappa(\sB_1) \geq \kappa(\sL_1) \geq \Var(f)$, 
where the latter inequality holds again by using~\cite[Thm~1.1]{Kawamata85}.

Furthermore, by flat base change, the two line bundles $\sL_1$ and $\sB_1$ are isomorphic 
away from $\disc(f_1)$ and thus so are $\sL_1$ and $\gamma^*\sL$.
Now, as $\sL_1$ is a $G$-sheaf, by using the structure of $\gamma$ and the fact that 
$\sB$ and $\sL_1$ differ only over $\gamma^{-1}\disc(f)$, 
we can also see that there is a line bundle $\sA$ on $X$ such that 
$\sL_1^d \cong \gamma^*\sA$, for sufficiently divisible integer $d$.

Our aim is now to show that $\sL$ and $\sL_1$ admit the required injections~\ref{eq:include}.
To this end, we consider the injections  

$$
\sL_1^m \subseteq  \sB_1^m   \hooklongrightarrow  \bigotimes^{tm}  (f_1)_* \omega^m_{Y_1/X_1}.
$$
We set $r:= tm$. Again, using the same arguments as above, 
we find 

$$
\sL_1^m  \subseteq  \sB_1^m  \subseteq  (f_1^{(r)})_* \omega^m_{Y_1^{(r)}/ X_1},
$$
verifying the inclusion \ref{eq:include} for $\sL_1$.
On the other hand, as before we have 

$$
\big( f_1^{(r)} \big)_*  \omega^m_{Y_1^{(r)}/X_1}  \subseteq  
    \gamma^* \big( f^{(r)}_* \omega^m_{Y^{(r)}/X} \big)
$$
so that 
$$
\gamma^* (\sL^m) \cong \sB_1^m  \hooklongrightarrow \gamma^* \big( f_*^{(r)} \omega^m_{Y^{(r)}/X}  \big).
$$
By applying the $G$-invariant section functor $\gamma_*(\cdot)^G$
to both sides we the required injection 

$$
\sL^m \hooklongrightarrow  f_*^{(r)}  \omega^m_{Y^{(r)}/X}.
$$
\end{proof}

\medskip

\subsection{Hodge theoretic constructions}\label{subsect:Hodge}
Let $f: Y \to X$ be a surjective, projective morphism of smooth quasi-projective varieties 
$Y$ and $X$ of relative dimension $n$ and let $D=\disc(f)$ be the discriminant locus of $f$. Set $\Delta: = f^*(D)$
and assume that the support of $D$ and $\Delta$ is simple normal 
crossing. 

\begin{definition}[Systems with $\sW$-valued operators]
Let $\sW$ be a coherent sheaf on $X$. We call the graded torsion free sheaf
$\sF = \bigoplus \sF_{i}$ a system with $\sW$-valued operator, if it can be 
equipped with a sheaf homomorphism $\tau: \sF \to \sW \otimes \sF$
satisfying the Griffiths transversality condition 
$\tau|_{\sF_{i}} : \sF_{i}  \to \sW \otimes  \sF_{i+1}$.
\end{definition}

Throughout the rest of this section we will allow ourselves to 
discard closed subsets of $X$ of $\codim_X\geq 2$, whenever necessary.

Our goal is to construct a system $(\sF, \tau)$ with $\Omega^1_X(\log D)$-operator $\tau$
which can be equipped with compatible maps to a system of Hodge bundles underlying 
the variation of Hodge structures of a second family that arises from $f$ via certain covering 
constructions.

To this end, let $\sM$ be a line bundle on $Y$ with $H^0(Y, \sM^m) \neq 0$.
Let $\psi: Z\to Y$ be a desingularization of the finite cyclic covering 
associated to taking roots of a non-zero section $s\in H^0(Y, \sM^m)$
so that 

\begin{equation}\label{eq:NV}
H^0(Z, \psi^*\sM)\neq 0,
\end{equation}
cf.~\cite[Prop.~4.1.6]{Laz04-I}. 

Now, consider the exact sequence of relative differential forms

\begin{equation}\label{eq:exact}
\xymatrix{
0  \ar[r]  &  f^*(\Omega^1_X (\log D))  \ar[r]   & \Omega^1_Y (\log \Delta )  \ar[r]  &  \Omega^1_{Y/X}(\log\Delta) \ar[r]  &  0,
}
\end{equation}
which is locally free over $X$ (after removing a codimension two subset of $X$).
Define $h:= f\circ \psi$.

The bundle $\Omega^{p}_Y\log(\Delta)$
can be filtered by a decreasing filtration $F^i$, $i\geq 0$, with 
$F^i /F^{i+1} \cong \Omega_{Y/X}^{p -i}(\log\Delta)) \otimes f^*(\Omega^i_X(\log D) $.
After taking the quotient by $F^2$, the exact sequence 
$0\to F_1 \to F_0 \to F^0 / F^1 \cong   \Omega^{p}_{Y/ X}(\log\Delta) \to 0$
reads as

\begin{equation}\label{eq:long}
\xymatrix{
0 \ar[r] &  f^*( \Omega^1_X(\log D)) \otimes \Omega_{Y/X}^{p -1} \log(\Delta)  \ar[r]  & F^0/F^2  \ar[r] &  \Omega^{p}_{Y/X} (\log\Delta)  \ar[r]  &    0,
}
\end{equation}
where $( \Omega_{Y/X}^{p -1} (\log\Delta)   \otimes f^*( \Omega^1_X(\log D))\cong F^1/F^2$
and 

\begin{equation}\label{eq:1st}
F^0/F^2  \cong  \Omega^{p}_Y (\log\Delta)/ \big( f^* \Omega^2_X(\log D) \otimes 
\Omega^{p-2}_{Y/X}(\log\Delta) \big).
\end{equation}

We tensor the sequence (\ref{eq:long}) by $\sM^{-1}$ to get the 
exact sequence

\begin{equation*}
\begin{multlined}
A_{p}: \; \; 0 \to   f^*( \Omega^1_X(\log D))\otimes \Omega_{Y/X}^{p -1} \otimes \sM^{-1}   \to 
(F^0/F^2)\otimes \sM^{-1} \to    \\
\;  \; \; \; \; \; \; \; \; \; \; \; \; \; \; \; \; \; \longrightarrow \Omega^{p}_{Y/X} \log(\Delta) \otimes \sM^{-1} \to 0.
\end{multlined}
\end{equation*}
Now, define a $\Omega^1_X(\log D)$-valued system $(\sF = \bigoplus \sF_{i}, \tau)$ of weight $n$ on $X$ by

\begin{equation}\label{eq:Fsystem}
\sF_{i} : =  \R^{i} f_* \big( \Omega^{n-i}_{Y/X}  (\log\Delta)  \otimes \sM^{-1}  \big) / {\text{torsion}},
\end{equation}
whose map $\tau|_{\sF_{i}}: \sF_{i} \to \Omega^1_X(\log D)\otimes  \sF_{i+1}$ is determined by
the corresponding connecting map in the long exact cohomology sequence arising from $\dR f_{*}(A_{n-i})$
\footnote{With $\tau$ being defined by connecting 
homomorphisms, $(\sF, \tau)$ is not necessarily a Higgs bundle.}, where $\dR f_{*}(A_{n-i})$
is a distinguished triangle in the bounded derived category of coherent sheaves.

Next, let $\disc(h) = (D+S)$. By removing a subset of $X$ of $\codim_X\geq 2$
we may assume that $(D+S)$ has simple normal crossing support. We further assume  
that, after replacing $Z$ by a higher birational model, if necessary, the divisor $h^*(D+S)$ has simple normal crossing support, 
which
we denote by $\Delta'$. 

Similar to the above construction, we can consider the exact sequence 

$$
B_{p} : \; \;    h^*\big(\Omega_X^1(\log(D+S) ) \big)  \otimes \Omega^{p -1}_{Z/X} ( \log\Delta')  
  \to G^0/{G^2}  \to    \Omega^{p}_{Z/X}(\log\Delta') \to 0,
$$
with $G^i$ being the decreasing filtration of $\Omega^p_Z(\log \Delta')$ associated to the sequence 

$$
\xymatrix{
0 \ar[r] &  h^*\Omega^1_X (\log(D+S)) \ar[r]  & \Omega^1_Z(\log \Delta') \ar[r] &  \Omega^1_{Z/X} (\log \Delta') \ar[r]  &  0.
}
$$
In this case we have 

\begin{equation}\label{eq:2nd}
G^0 / G^2  \cong \Omega_Z^{p}(\log \Delta')  /  \big(  h^*\Omega^2_X\log(D+S)  \otimes \Omega_{Z/X}^{p-2}
(\log\Delta')    \big).
\end{equation}

Now, let $(\sE, \theta)$ with $\theta: \sE \to \sE \otimes \Omega^1_X (\log(D+S))$ 
be the logarithmic Higgs bundles
underlying the canonical extension of the local system 
$\R^n h_*(\bC_{Z\smallsetminus \Delta'})$, whose residue has eigenvalues contained in $[0,1)$, 
cf.~\cite[Prop.~I.5.4]{Deligne70}. According to Steenbrink~\cite[Thm.~2.18]{Ste76} and Katz-Oda~\cite{KO68}
the associated graded module of the Hodge filtration induces  
the structure of a system of Hodge bundles on $(\sE, \theta)$ whose graded pieces $\sE_{i}$ are given by 

\begin{equation}\label{eq:Esystem}
\sE_{i} = \R^{i} h_* \big(  \Omega^{n-i}_{Z/X} (\log \Delta') \big),
\end{equation}
with morphisms $\theta: \sE_{i} \to   \Omega^1_X (\log(D+S)) \otimes \sE_{i +1}$ 
being determined by those in the cohomology sequence arising from $\dR h_*(B_{n-i})$.

Our aim is now to construct a morphism 
from $(\sF, \tau)$ 
to $(\sE, \theta)$ that is compatible with $\tau$ and $\theta$. To this end, we observe that the non-vanishing (\ref{eq:NV})
implies that there is a natural injection 

$$
\psi^*\big(  \Omega^{n-i}_{Y/X}  (\log\Delta) \otimes \sM^{-1}  \big)    
       \hooklongrightarrow  \Omega^{n-i}_{Z/X} (\log\Delta').
$$
Together with the two isomorphisms (\ref{eq:1st}) and (\ref{eq:2nd}) it then follows that the induced 
map 

$$
\psi^*\big(\frac{F^{n-i}_0}{F^{n - i}_2} \otimes \sM^{-1} \big)  \longrightarrow  \frac{G^{n-i}_0}{G^{n-i}_2}
$$
is an injection. Therefore, the sequence defined by $\psi^*A_{n-i}$ is a subsequence of $B_{n-i}$,
inducing a morphism $A_{n-i}  \longrightarrow  \dR \psi_*(B_{n-i})$ (in the bounded derived category of coherent sheaves). 
This leads 
to the map

$$
\dR f_*  (A_{n-i})  \longrightarrow \dR h_* (B_{n-i}).
$$
In particular there is a sheaf morphism

$$
\Phi_{i}: \underbrace{\R^{i} f_*  \big( (\Omega^{n-i}_{Y/X} (\log\Delta)  \otimes \sM^{-1}  \big)/{\text{torsion}}}_{\sF_{i}}
  \longrightarrow \underbrace{\R^{i}  h_*  \big( \Omega^{n-i}_{Z/X}  (\log\Delta')  \big)}_{\sE_{i}}.
$$
The compatibility of $\Phi_{i}$, with respect to $\tau$ and $\theta$, follows from the fact that, 
by construction, each $\tau|_{\sF_{i}}$ is defined by the corresponding 
map in the cohomology complex of $\dR f_*(A_{n-i})$;

$$
\xymatrix{
\dR f_*(A_{n-i})/\text{torsion} : & ...  \ar[r]  &  \sF_{i}  \ar[rr]^(0.35){\tau} \ar[d]^{\Phi_{i}}  &&  
                                                  \Omega^1_X(\log D)\otimes  \sF_{i+1}   \ar[d]^{\Phi_{i+1}}  \ar[r] & ...\\
\dR h_*(B_{n-i}) : & ...  \ar[r]  &  \sE_{i}  \ar[rr]^(0.3){\theta}        &&    
         \Omega^1_X (\log(D+S))\otimes  \sE_{i+1}  \ar[r]  &  ... \; .
}
$$
Thus $\bigoplus \Phi_i$ is a map of 
systems, 
with its image being a subsystem of Hodge sheaves. 

Furthermore, as a direct consequence of the fact that $\Phi_0$ is induced 
by the injective map 

$$
\psi^* \big(  \Omega^n_{Y/X} (\log \Delta) \otimes \sM^{-1} \big)  \longrightarrow  \Omega^n_{Z/X}(\log \Delta')
$$
and the natural injection $\sO_Y \longrightarrow \psi_* \sO_Z$ we find that $\Phi_0$ is injective. For future
reference we record this last observation in the following proposition.

\begin{proposition}\label{prop:injective}
The sheaf morphism $\Phi_0:  \sF_0  \longrightarrow \sE_0$ is injective.
\end{proposition}

\medskip

\section{Constructing pseudo-effective Viehweg-Zuo subsheaves}
\label{sect:Section3-VZ}

In the current section we will prove Theorem~\ref{thm:VZ}. We will be working 
in the context of the following set-up.

\begin{set-up}\label{setup}
Let $f: Y\to X$ be a smooth compactification of the smooth projective family
$f_U: U\to V$ whose general fiber has a good minimal model and $\Var(f) \neq 0$.
Set $D$ to be the divisor defined by $X\smallsetminus D\cong V$ 
and $\Delta = \Supp(f^*D)$ (both $D$ and $\Delta$ are assumed to have simple 
normal crossing support). Let $\gamma: X_1 \to X$ be the cyclic map
associated to the semistable reduction $f_1: Y_1 \to X_1$ of 
$f: Y\to X$ in codimension one. Let $D_1 =\disc(f_1)$ and set $\Delta_1$ to be
the maximal reduced divisor supported over $f_1^* D_1$.
Finally, let $\sL_1$ be the line bundle on $X_1$, with $\kappa(X_1, \sL_1)\geq \Var(f)$, 
as defined in Proposition~\ref{prop:positivity}.

\end{set-up}

\begin{proposition}\label{prop:Hodge}
In the situation of Set-up~\ref{setup}, 
after removing a subset of $X_1$ of $\codim_{X_1}\geq 2$,
the following constructions and properties can
be verified.

\begin{enumerate}
\item\label{item:H1} There exists a system of logarithmic Hodge bundles $(\sE= \bigoplus \sE_{i}, \theta)$ on $X_1$
with $\theta: \sE_{i} \to \Omega^1_{X_1} (\log (D_1+S)) \otimes \sE_{i + 1}$, where $S\subset X_1$ 
is a reduced, normal crossing divisor.
\item \label{item:H2} The torsion free sheaf $\ker(\theta|_{\sE_{i}})$ is seminegatively curved.
\item\label{item:H3} There exists a subsystem $(\sG = \bigoplus \sG_{i}, \theta)$ of $(\sE, \theta)$
such that $\theta(\sG_{i}) \subseteq  \gamma^*\big(\Omega^1_X(\log D)\big) \otimes \sG_{i+1}$.
\item \label{item:H4} There is an injection 
$\sL_1 \hooklongrightarrow \sG_0$.
\end{enumerate}
\end{proposition}

Here, and following the terminology introduced in~\cite[p.~357]{MR2449950} by Berndtsson and P\u{a}un, we say that a 
torsion free sheaf $\sN$ is seminegatively curved if it carries a seminegatively 
curved (possibly singular) metric $h$ over its smooth locus $X_{\sN}$ (where $\sN$ is locally free).
We refer to~\cite{MR2449950} for more details (see also the survey paper~\cite{Pau16}).
An immediate consequence of this property is the fact that, once it is satisfied, then 
$\det(\sN)$ extends to a anti-pseudo-effective line bundle on the projective variety $X$
and this is all that is needed for our purposes in the current paper. 

Granting Proposition~\ref{prop:Hodge} for the moment let us proceed with the proof of Theorem~\ref{thm:VZ}.

\medskip

\noindent
\emph{Proof of Theorem~\ref{thm:VZ}.}
Let $X_1^\circ\subseteq X_1$ be the open subset over which Items~\ref{item:H1}--\ref{item:H4} 
in Proposition~\ref{prop:Hodge} are valid and set $X^\circ$ to be it image under $\gamma$
in $X$. By iterating the morphism 
$$
\id \otimes \theta :  \big( \gamma^* \Omega^1_{X^\circ} (\log D) \big)^{\otimes j}  \otimes \sG_j      \longrightarrow   
     \big( \gamma^* \Omega^1_{X^\circ} (\log D) \big)^{\otimes j+1}\otimes    \sG_{j+1} ,
$$
for every $j \in \bN$, we can construct a map 

$$
\theta^j :  \sG_{0}  \longrightarrow  \big( \gamma^* \Omega^1_{X^\circ} (\log D) \big)^{\otimes j} \otimes \sG_j  .
$$

We note that $\theta^j(\sG_0)\neq 0$. Otherwise, there is an injection

$$
\sL_1|_{X_1^\circ}  \hooklongrightarrow   \ker(\theta|_{\sG_0}).
$$
On the other hand, 
by construction we have $\ker(\theta|_{\sG_0})  \subseteq  \ker(\theta|_{\sE_0})$
and 
according to Item~\ref{item:H2}
$\ker(\theta|_{\sE_0})$ is seminegatively curved. This implies that 
$\sL_1$ is anti-pseudo-effective and therefore $\kappa(X_1, \sL_1)\leq 0$,
contradicting our assumption on $\Var(f)$ not being equal to zero.

Now, let $k$ be the positive integer defined by $k: = \max\{ j \in \bN  \;  |  \;  \theta^j(\sG_0) \neq 0  \}$, 
so that 
$$
\theta^k (\sG_0)  \subset  \big(\gamma^* \Omega^1_{X^\circ} (\log D)  \big)^{\otimes k} \otimes
     \underbrace{\ker(\theta|_{\sG_k})}_{\sN_k}    .
  $$
From Item~\ref{item:H4} it follows that there is a non-trivial morphism 

$$
\sL_1 \longrightarrow     \big(\gamma^* \Omega^1_{X^\circ} (\log D)  \big)^{\otimes k} \otimes \sN_k  ,
$$
which, by Item~\ref{item:H2}, implies the existence of a non-zero map 

\begin{equation}\label{eq:FinalInject}
  \sL_1^{s} \otimes   \big(\det(\sN_k) \big)^{-1}   \longrightarrow   \big(\gamma^* \Omega^1_{X^\circ} (\log D)  \big)^{\otimes k'},
\end{equation}
for $s:=\rk(\sN_k)$ and $k':= k\cdot s$. Now, let $\sB_1$ be the line bundle on $X_1$ defined by the extension of 
$\det(\sN_k)^{-1}$ so that (\ref{eq:FinalInject}) extends to the injection 

$$
\sL_1^s \otimes \sB_1    \hooklongrightarrow   \big(\gamma^* \Omega^1_{X} (\log D)  \big)^{\otimes k'}.
$$

We now verify that $\sB_1$ is pseudo-effective. 
Again, according to Item~\ref{item:H2}, the torsion free sheaf 
$\sN_{k} \subseteq \ker(\theta|_{\sE_k})$ is seminegatively curved and therefore so is $\det(\sN_k)$ 
which then extends to a
anti-pseudo-effective line bundle $\sB_1^{-1}$ on $X_1$.  

To finish the proof of Theorem~\ref{thm:VZ}, note that according to Item~\ref{item:down2} in
Proposition~\ref{prop:positivity} there is a line bundle $\sA$ on $X$ such that 
$\sL_1^d\cong \gamma^* \sA$, for some positive integer $d$.
With no loss of generality we may assume that the exponent $d$ is equal to one.
We consider the injection 

$$
\sB_1   \hooklongrightarrow \sL_1^{-s}  \otimes   \big(\gamma^* \Omega^1_{X} (\log D)  \big)^{\otimes k'} 
 \cong  \gamma^* \sA^{-s}  \otimes \big(\gamma^* \Omega^1_{X} (\log D)  \big)^{\otimes k'}
$$
whose image we may assume to be saturated. This us to the injection 

\begin{equation}\label{eq:star}
\bigotimes_{g\in G}  g^*  \sB_1  \hooklongrightarrow  \gamma^*\big( \sA^{-s}  \otimes \Omega^1_{X} (\log D) ^{\otimes k'} \big)^{\otimes |G|},
\end{equation}
where $G: =\Gal(X_1/X)$.

Now, the line bundle on the left hand side of (\ref{eq:star}) is a pseudo-effective 
$G$-sheaf. As before, thanks to~\cite[Thm.~4.2.15]{MR2665168}, it follows that there is a 
line bundle $\sB$ on $X$ such that $\gamma^*\sB$ is isomorphic to the left 
hand side of (\ref{eq:star}). It is not difficult to see that $\sB$ is also pseudo-effective
and therefore, by applying $\gamma_*(\cdot)^G$ to both sides of (\ref{eq:star}), 
we have an injection 

$$
\sA ^{s\cdot |G|}  \otimes \sB    \hooklongrightarrow   \big(  \Omega^1_X (\log D)  \big)^{\otimes k' \cdot |G|},
$$
as required.
\qed

\medskip

\noindent
\emph{Proof of Proposition~\ref{prop:Hodge}.}
The proof is a consequence of Proposition~\ref{prop:positivity} combined with the Hodge theoretic 
constructions in Subsection~\ref{subsect:Hodge}. To lighten the notation we will 
replace the initial family $f: Y\to X$ in Set-up~\ref{setup} by $f^{(r)}: Y^{(r)}\to X$
and $f_1: Y_1 \to X_1$ by $f_1^{(r)}: Y_1^{(r)} \to X_1$, 
which were constructed in Proposition~\ref{prop:positivity}.

After removing a codimension two subset from $X_1$ over which the 
inclusion \ref{eq:include} holds, define the line bundle 
$$
\sM_1 = \omega_{Y_1/X_1}\otimes f_1^*(\sL_1^{-1})
$$
so that $H^0(Y, \sM_1^m ) \neq 0$.  The arguments in Subsection~\ref{subsect:Hodge}
can now be used to construct the two systems $(\sE, \theta)$ and $(\sF, \tau)$
defined in~(\ref{eq:Fsystem}) and~(\ref{eq:Esystem}), with logarithmic poles along
$(D_1+S)$ and $D_1$, respectively. Here after deleting a codimension two subset we have assumed
that $\Supp(D_1+S)$ is simple normal crossing. 

Item~\ref{item:H2} follows from Zuo's result \cite{Zuo00}---based on 
Cattani, Kaplan and Schmid's work on asymptotic behaviour of Hodge 
metrics, cf.~\cite{CKS} and \cite[Lem~3.2]{PTW}---and Brunbarbe~\cite{Bru15},
and more generally from Fujino and Fujisawa \cite{FF17}.
The reader may wish to consult Simpson~\cite{MR1040197} and~\cite{Bru17}
where these problems are dealt with in the more general setting of tame harmonic metrics.

\begin{claim}\label{claim:pullback}
For the restricted Higgs field $\tau|_{\sF_{i}}: \sF_{i} \longrightarrow \sF_{i +1}  \otimes \Omega^1_{X_1}(\log D_1)$
we have $\tau(\sF_{i}) \subseteq \sF_{i+1} \otimes \gamma^*\big( \Omega_X^1 (\log D) \big)$.
\end{claim}

\noindent
\emph{Proof of Claim~\ref{claim:pullback}.}
We define $\sM:= \omega_{Y/X} \otimes f^*\sL^{-1}$ ($\sL$ being the line bundle in 
Proposition~\ref{prop:positivity}) and repeat the constructions in Subsection~\ref{subsect:Hodge}
for $(\sF, \tau)$ in the complement of $D$ to construct 
the system $(\sF_{i}', \tau')$ over $X\smallsetminus D$
defined by 

$$
\sF_{i}'  =  \R^{i} f_*  \big(   \Omega^{n-i}_{Y/X} (\log D) 
                                                   \otimes \sM^{-1}   \big) /\text{torsion} \; |_{X\smallsetminus D}.  
$$
Now, by flat base change and the isomorphism
$\sL_1|_{X_1\smallsetminus D_1} \cong (\gamma^*\sL)|_{X_1\smallsetminus D_1}$ 
(Proposition \ref{prop:positivity}, Item~\ref{item:down1}) we find that the inclusion 

$$
\gamma^*(\sF', \tau')  \subseteq  (\sF, \tau)|_{X_1\smallsetminus D_1}
$$
holds as systems and that $\gamma^* \sF'_{i}\cong \sF_{i}|_{X_1\smallsetminus D_1}$.
Consequently, the map 

$$
\tau: \sF_{i}   \longrightarrow   \Omega^1_{X_1} (\log D_1) \otimes \sF_{i+1} 
$$
factors through the inclusion $  \gamma^*\Omega^1_X(\log D)  \otimes \sF_{i+1}
\subseteq   \Omega^1_{X_1}(\log D_1)  \otimes \sF_{i +1} $. This proves the claim.  \qed

For~\ref{item:H3}, define $(\sG = \bigoplus \sG_{i}, \theta) \subset (\sE, \theta)$
to be the image of the system $(\sF= \bigoplus \sF_{i}, \tau)$ under the 
morphism $\Phi_{i}$, constructed in Subsection~\ref{subsect:Hodge}.
Item~\ref{item:H3} now follows from Claim~\ref{claim:pullback}.

It remains to verify~\ref{item:H4}. As $f_1$ is semistable in codimension one,
over this locus we have $\omega_{Y_1/X_1} \cong \Omega^n_{Y_1/X_1}(\log \Delta_1)$.
Therefore we have 

$$
\sF_0 \cong  \R^0 (f_1)_* \big(  \Omega^n_{Y_1/X_1} (\log \Delta_1)  \otimes ( \omega_{Y_1/X_1}^{-1}  \otimes f_1^*\sL_1 )  \big)
\cong \sL_1.
$$
The required injection now follows from the fact that, according to Proposition~\ref{prop:injective},
the sheaf morphism $\Phi_0$ is injective.\qed

\begin{remark}
We note that in the maximal variation case a simpler argument may be used to establish 
Proposition~\ref{prop:Hodge} directly on $X$ (in particular Claim~\ref{claim:pullback} and 
Items~\ref{item:down1} and~\ref{item:down2} in 
Proposition~\ref{prop:positivity} are superfluous). Briefly, the reason is
in this case with $\det f_* \omega^m_{Y/X}$ being big, we can always find a sufficiently ample 
line bundle $\sA$ on $X$ such that $\sA(-D)$ is ample and that 
there is an injection $\sA^m \to f_*^{(r)}  \omega^m_{Y^{(r)}/X}$, for sufficiently large and divisible
$m$. Thus, after replacing $f$ by $f^{(r)}$, we can repeat the construction 
of $(\sF, \tau)$ and $(\sE, \theta)$ in Subsection~\ref{subsect:Hodge},
with $\sM: = \omega_{Y/X}\otimes f^*\sA^{-1}$. By construction we then find
an injection 

$$
\sA(-D)  \longrightarrow   \sF_0  .
$$
With $\Phi_0$ being injective, the rest now follows as before.

\end{remark}

\medskip

\section{Vanishing results}
\label{sect:Section4-Birational}

In this section we will prove Theorem~\ref{thm:vanishing}. 
The methods will heavily rely on birational techniques and results in
the minimal model program. For 
an in-depth discussions of preliminary notions and background we refer 
the reader to the book of Koll\'ar and Mori \cite{KM98} and the 
references therein. 

Not surprisingly an important construction that we will repeatedly make 
use of is the Iitaka fibration. Please see \cite[Sect.~1]{Mori87} and \cite[Sect.~2.1.C]{Laz04-I}
for the definition and a review of the basic properties.

\begin{notation}
Let $\sL$ be a line bundle on a normal projective variety $X$ with $\kappa(X, \sL)>0$.
By $\phi^{(I)}: X^{(I)}\to Y^{(I)}$ we denote the Iitaka fibration of $\sL$ with 
an induced birational morphism $\pi^{(I)}: X^{(I)} \to X$.
\end{notation}

\subsection{Proof of Theorem~\ref{thm:vanishing}}

We begin by stating the following lemma concerning the behaviour of the Kodaira dimension 
on fibers of the Iitaka fibration. 
The proof follows from standard arguments; see for example~\cite[pp.~136--137]{Laz04-I}.

\begin{lemma}\label{lem:effective}
Let $X$ be a smooth projective variety and ${\pi^{(\I)}}: X^{(\I)} \to X$ a birational morphism such that
$\phi^{(\I)}: X^{(\I)} \to Z^{(\I)}$ is
the Iitaka fibration of the line bundle $\sL$ on $X$.
Then, for any ${\pi^{(\I)}}$-exceptional and effective divisor $E$, we have 

$$
\kappa\big(F^{(\I)}, ({\pi^{(\I)}}^*(\sL) \otimes \sO_{X^{(\I)}}(E))\big|_{F^{(\I)}}\big) = 0,
$$ 
where $F^{(\I)}$ is a very general 
fibre of $\phi^{(\I)}$.
\end{lemma}

The next lemma is the final technical background that 
we need before we can proceed to the proof of Theorem~\ref{thm:vanishing}.
It relies on the so-called 
flattening lemma, due to Gruson and Raynaud, which we recall below. 

Let $f: X\to Z$ be an algebraic fiber space of normal (quasi-)projective 
varieties. There exists an equidimensional fiber space $f' : X' \to Z'$
of normal varieties that is birationally equivalent to $f$ through 
birational morphisms $\sigma : X' \to X$ and $\tau: Z' \to Z$. We call
$f'$ the \emph{flattening} of $f$.

\begin{lemma}\label{lem:flat}
Let $f: X\to Z$ be a fiber space of normal projective varieties with a 
flattening $f' : X' \to Z'$ as above. Let $A$ be a $f$-nef and effective 
$\bQ$-divisor and assume that $A|_F\equiv 0$, where $F$ is the 
general fiber of $f$. There exists $A_{Z'}\in \Div(Z')_{\bQ}$
such that $\sigma^*(A) \sim_{\bQ}  (f')^*(A_{Z'})$.
\end{lemma}

\begin{proof}
It suffices to show that $\sigma^*(A)$ is $f'$-numerically trivial. Aiming for 
a contradiction, assume that there exists a fiber $F_{z'}$ of $f'$ containing 
an irreducible contractible curve $C$ such that $\sigma^*(A)\cdot C  \neq 0$.
Let $A_1, \ldots, A_{d-1}$ and $A_d$ be a collection of ample divisors in 
$X'$ such that $(H_1\cdot \ldots \cdot H_d)\cap F_{z'}$ defines 
the numerical cycle of $(C+\widetilde C)$, where $\widetilde C$ is an effective 
$1$-cycle in $F_{z'}$. Since $\sigma^*(A)|_{F'}\equiv 0$, where $F'$
is a general fiber of $f'$, we have 

\begin{equation}\label{eq:intersect}
\sigma^*(A) \cdot (C+ C') =0.
\end{equation}

On the other hand, $\sigma^*(A)$ is $f'$-nef. Combined with (\ref{eq:intersect}), this 
implies that $\sigma^*(A)\cdot C = 0$, which contradicts our initial assumption. 

\end{proof}

\bigskip

\noindent
\emph{Proof of Theorem~\ref{thm:vanishing}. (Preparation).}
Let $\sL$ and $\sB$ be two line bundles on $X$, with $\sB$ being pseudo-effective,
such that, for some $i\in \bN$, there is a non-trivial morphism 
$$
\sL \otimes \sB  \longrightarrow   (\Omega^1_X(\log D))^{\otimes i}.
$$

\begin{claim}\label{claim:CP}
The line bundle $\omega_X^i(D) \otimes \sL^{-1}$ is pseudo-effective.
\end{claim}

\noindent
\emph{Proof of Claim~\ref{claim:CP}.}
Let $\sF$ be the saturation of the image of the non-trivial morphism 
$$
\sL \otimes \sB   \longrightarrow (\Omega^1_X(\log D))^{\otimes i}
$$
and $\sQ$ the torsion-free quotient resulting in the exact sequence
$$
\xymatrix{
0 \ar[r] &  \sF \ar[r] &  (\Omega^1_X(\log D))^{\otimes i} \ar[r] &  \sQ  \ar[r]   &  0.
}
$$
After taking determinants we find that
\begin{equation}\label{eq:CP}
\omega_X^i(D)  \otimes \sF^{-1}  \cong   \det(\sQ).
\end{equation}
Thanks to~\cite[Thm.~1.3]{CP16} the right-hand side of (\ref{eq:CP}) is pseudo-effective
and thus so is the left-hand side. 
This implies that $\omega^i_X(D)\otimes (\sL \otimes \sB)^{-1} \in \overline{\NE}^1(X).$
But $\sB$ is also assumed to be pseudo-effective and therefore  
$\omega^i_X(D)\otimes \sL^{-1}\in \overline{\NE}^1(X)$. This finishes the proof of the claim.

\smallskip
Before we proceed further, notice that
we may assume that $(i(K_X+D)+L)$ is not big, where $\sL = \sO_X(L)$.
Otherwise the divisor $K_X+D$ is big, as it can be written
as the sum of a pseudo-effective and big divisors: 
$$
K_X + D  =  \frac{1}{2i}   \big( ( i\cdot(K_X+D) +L )  + ( i\cdot( K_X+D ) - L  )\big).
$$
We can also assume that $\kappa(\sL) \geq 1$.

The first step in proving the theorem consists of replacing $X$ by a birational model
$Y$ where establishing the vanishing in Theorem~\ref{thm:vanishing}
proves to be easier.

\begin{claim}\label{claim:KeyClaim}
There exists a birational morphism $\pi : Y\to X$ from a smooth projective variety $Y$ that can
be equipped with a fiber space $f: Y \to Z$ over a smooth projective 
variety $Z$. Furthermore, $Y$ contains a reduced divisor $D_Y$ such that $(Y,D_Y)$ is log-smooth 
and a divisor $L_Y$
which satisfy the following properties.

\begin{enumerate}
\item \label{item:1k} $\dim(Z) = \kappa(i(K_X+D)+L)$,
\item \label{item:2k} $\sO_Y(L_Y) \otimes \pi^*(\sB) \subseteq \big(\Omega^{1}_Y\log(D_Y)\big)^{\otimes i}$,
\item \label{item:3k} $\kappa(K_Y+D_Y) = \kappa(K_X+D)$,
\item \label{item:4k} $\kappa(L_Y) = \kappa(L)$,
\item \label{item:5k} $\kappa(F, (i(K_Y+D_Y) + L_Y)|_{F}) = 0$ and $\kappa(F, (K_Y+D_Y)|_F)=0$, where $F$ is a very general fiber of $f$,
\item \label{item:6k} $\big(  i(K_Y+D_Y) - L_Y \big) \in \overline{\NE}^1(Y)$, and 
\item \label{item:7k} $L_Y \sim_{\bQ} f^*(L_Z)$, for some $L_Z\in \Div_{\bQ}(Z)$.
\end{enumerate}
\end{claim}

\noindent
\emph{Proof of Claim~\ref{claim:KeyClaim}.}
Let $\psi : X \dashrightarrow Z$ be the rational 
mapping associated to the linear system $|m\cdot(i(K_X+D)+L)  |$,
with $m$ being sufficiently large so that $\dim(Z) = \kappa(i(K_X+D)+L)$.
Note that as $\kappa(L)\geq 1$, we have $\dim(Z) \geq 1$. Let $\psi_1: X_1 \to Z_1$
be the Iitaka fibration of $(i(K_X+D)+L)$ resulting in the commutative diagram 

$$
\xymatrix{
X_1  \ar[d]_{\pi_1}  \ar[rr]^{\psi_1}   &&   Z_1  \ar@{-->}[d]  \\
X \ar@{-->}[rr]^{\psi}  &&   Z,
}
$$
where $\pi_1: X_1 \to X$ is a birational morphism. Define $L_1: = \pi_1^*(L)$.
Let $E_1$ and $E_2$ be two effective and exceptional divisors such that

$$
K_{X_1}  + \underbrace{\widetilde D  +  E_2}_{: = D_1}  \sim \pi_1^*(K_X+D)  + E_1,
$$
where $\widetilde D$ is the birational transform of $D$. Note that Claim~\ref{claim:CP}
implies that $\big( i\cdot (K_{X_1} + D_1 )  - L_1 \big)\in \overline{\NE}^1(X_1)$.
Furthermore, by Lemma~\ref{lem:effective} we have $\kappa(F_1, (i(K_{X_1} +D_1) + L_1)|_{F_1}) = 0$,
where $F_1$ is a very general fibre of $\psi_1$. On the other hand, we have

$$
\kappa(F_1, L_1|_{F_1})  \leq  \kappa( F_1 , (i(K_{X_1} + D_1) + L_1 )|_{F_1}) =0.
$$
Therefore, we find that $\kappa(F_1, L_1|_{F_1}) =0$. As a result, and thanks to~\cite[Def--Thm.~1.11]{Mori87},
the Iitaka fibration $\psi_2: X_2 \to Y$ of $L_1$ factors through the fiber space $\psi_1: X_1 \to Z_1$ 
via a birational morphism $\pi_2: X_2 \to X_1$ and a
rational map $\nu:  Z_1 \dashrightarrow Y$ (see Diagram~\ref{eq:diag2} below).
Let $\widetilde \nu:  Z_2 \to Y$ be a desingularization
of $\nu$ through the birational morphism $\mu:  Z_2 \to Z_1$. Finally, let $\psi_3: X_3 \to Z_2$
be a desingularization of the rational map $X_2 \dashrightarrow Z_2$,
defined by the composition of $\pi_2$, $\psi_1$ and $\mu^{-1}$ (where it is defined),
via the birational morphism $\pi_3:  X_3 \to X_2$.

\begin{equation}\label{eq:diag2}
 \begin{gathered}
 \xymatrix{
& X_3  \ar[d]_{\pi_3}   \ar@/^11mm/[ddrrrr]^{\psi_3} \ar@/_6mm/[ddl]_{\pi}  \\
&  X_2 \ar[d]_{\pi_2}     \ar[rr]^{\psi_2}  && Y  \\
X     & X_1    \ar[l]^{\pi_1}  \ar[rr]^{\psi_1} &&  Z_1 \ar@{-->}[u]^{\nu}  &&  Z_2   \ar[ll]_{\mu}  \ar[ull]_{\widetilde \nu}
}
 \end{gathered}
\end{equation}

By construction, there is an effective $\bQ$-divisor $E\subset X_2$ 
and a very ample $\bQ$-divisor $L_Y$ in $Y$ such that 
$$
\pi_2^*(L_1)  - E \sim_{\bQ}    \psi_2^*(L_Y).
$$
Define $L_3: =  \pi_3^*\big(  \pi_2^*(L_1)   - E \big)$. Let $E_3$ and $E_4$ be two effective 
exceptional divisors for which we have

$$
K_{X_3}  + \underbrace{\widetilde D_3  + E_3}_{:= D_3}   \sim \pi_3^*\big( \pi_2^*(K_{X_1}+D_1)  \big) + E_4,
$$
where $\widetilde D_3$ is the birational transform of $D_1$.

\smallskip

We now claim that the two divisors $\big( i\cdot (K_{X_3}  + D_3) \big)$ and $L_3$ together with 
the fiber space $\psi_3: X_3  \to Z_2$ satisfy the properties listed in Claim~\ref{claim:KeyClaim} 
for $Y$, $D$, $L_Y$, $Z$ and $f$.

To see this, first note that $\kappa(L_3)= \kappa(L)$
and that 

$$
L_3 \sim \psi_3^*  \big( \underbrace{\widetilde \nu^*(L_Y)}_{:= L_{Z_3}} \big),
$$
thanks to the commutativity of Diagram~\ref{eq:diag2}. 

Next, to verify Property~\ref{item:5k}, let $F_3$ to be a very general fiber of $\psi_3$
and note that 

$$
0 \leq \kappa(F_3,  (K_{X_3} + D_3)\big|_{F_3})  =  \kappa(F_3,  (i(K_{X_3} + D_3) + L_3)\big|_{F_3}  ),
$$
On the other hand we have

\begin{equation}\label{eq:zero}
\kappa\big(F_3,  (i(K_{X_3}  +D_3) +L_3)|_{F_3} \big)  \leq  \kappa \big( F_3,  (i(K_{X_3} + D_3)  + \pi_3^*(\pi_2^*L_1))
  \big|_{F_3}  \big).
\end{equation}
The two equalities in \ref{item:5k} now follow form the fact that the right-hand side of (\ref{eq:zero})
is less than or equal to zero, cf. Lemma~\ref{lem:effective}.  

The pseudo-effectivity of $i\cdot (K_{X_3} + D_3)-L_3$ (Property~\ref{item:6k}) follows 
from the relation

$$
i\cdot (K_{X_3} + D_3)-L_3  \sim  \pi_3^* \big(  \pi_2^*( i\cdot(K_{X_1}+D_1) - L_1 )   \big)   
         +  \big( i\cdot E_4  + \pi_3^*(E) \big)
$$
and the fact that $\big(i\cdot (K_{X_1} + D_1)-L_1\big) \in \overline{\NE}^1(X_1)$.
The remaining properties hold by construction. 
This concludes the 
proof of Claim~\ref{claim:KeyClaim}.

\medskip

To lighten the notation, from now on we will assume that $i=1$.

\smallskip

\emph{Proof of Theorem~\ref{thm:vanishing}}.
After fixing the dimension of $X$, our proof will be based on induction on 
$d: = \kappa(K_X+D+L)$, assuming that $\kappa(L)>0$.
The next claim provides the base case.

\begin{claim}\label{claim:base}
If $d=1$, then $\kappa(\sL) = \kappa(\omega_X(D))$.
\end{claim}

\noindent
\emph{Proof of Claim~\ref{claim:base}.}
Using \ref{item:5k} we can see that the Iitaka fibration 
$\phi^{(I)}: Y^{(I)} \to Z^{(I)}$ of $(K_Y+D_Y+L_Y)$ factors through 
$f: Y\to Z$ via a birational morphism $\pi^{(I)} : Y^{(I)}\to Y$
and a finite morphism $\nu: Z\to Z^{(I)}$. As both $(\pi^{(I)}\circ f)$ 
and $\phi^{(I)}$ are fiber spaces, the finite map $\nu$ must
be trivial, that is the two maps $\phi^{(I)}$ and $f$ coincide. 
In particular we have $(K_Y+D_Y+L_Y)\sim_{\bQ} f^* (B_Z)$, for 
some very ample divisor $B_Z$ in $Z$. By using \ref{item:7k}
it then follows that 
\begin{equation}\label{eq:curve}
(K_Y+D_Y)  \sim_{\bQ}  \frac{1}{2} \big(    f^*(B_Z - 2 \cdot L_Z)  + f^*(B_Z)   \big),
\end{equation}
Now, as $(B_Z- 2\cdot L_Z) \in \overline{\NE}^1(Z)$ by \ref{item:6k}, we conclude that
the right-hand side of (\ref{eq:curve}) is ample in $Z$. 
Therefore $\kappa(K_Y+D_Y) = \kappa(L_Y) = 1$, which establishes the claim.

\medskip

\noindent
\emph{Inductive step.} We assume that Theorem~\ref{thm:vanishing}
holds for any line bundle $\sA$ on a smooth projective variety $W$ (having the same dimension 
as $X$) that satisfies the following two properties. 

\begin{enumerate}
\item There is a reduced divisor $D_W$ such that $(W, D_W)$ is log-smooth and 
$(\sA\otimes \sM) \subseteq \Omega^{\otimes i}_W\log(D_W)$, for some pseudo-effective line bundle $\sM$.
\item $\kappa(W, \omega_W(D_W) \otimes \sA) < d$.
\end{enumerate}

Let $(Y,D_Y)$, $L_Y$ and $f: Y\to Z$ be as in the setting of Claim~\ref{claim:KeyClaim}.
By the inductive step we may assume that $\dim(Z) = \kappa(K_Y+D_Y+L_Y)$.
Furthermore, we can use
Claim~\ref{claim:base} to exclude the possibility that 
$\kappa(K_Y+D_Y+L_Y)=1$.

We treat the case where the dimension of the fibers of $f$
is equal to $3$ (i.e. $\dim(Y) =5$ and $\kappa(K_Y+D_Y+L_Y) =2$).
The case of lower dimensional fibers can be dealt with similarly.

Let $g: (Y, D_Y) \dashrightarrow (Y_n , D_{Y_n})$ be the birational map associated 
to a relative minimal model program for $(Y,D_Y)$ over $Z$, consisting 
of $n$ number of divisorial and flipping contractions 

$$
g_j:  (Y_j, D_{Y_j})  \dashrightarrow (Y_{D_{j+1}}, D_{Y_{j+1}}),
$$
cf.~\cite[Chapt.~4]{SecondAsterisque}.
Here we have set $(Y_0, D_{Y_0}):= (Y,D_Y)$.
For each $1\leq j \leq n$, let $f_j: Y_{D_j} \to Z$ be the induced morphism resulting 
in the following diagram.

$$
\xymatrix{
Y \ar@{-->}@/^9mm/[rrrrrr]^{g} \ar@{-->}[rr]^{g_1} \ar[drrr]_{f}  && Y_1  \ar@{-->}[r]^{g_2}  \ar[dr]_(.30){f_1} &  .... 
     \ar@{-->}[r]^(0.38){g_{n-1}}  &  Y_{n-1} \ar@{-->}[rr]^{g_n} \ar[dl]^(.30){f_{n-1}}  &&   Y_n \ar[dlll]^{f_n}  \\
&&& Z.
}
$$

\begin{claim}\label{claim:1}
$\kappa(Y_n, K_{Y_n} + D_{Y_n} + f_n^*L_Z ) = \kappa(Y, K_Y+D_Y+L_Y)$.
\end{claim}

\begin{claim}\label{claim:2}
$\big( (K_{Y_n}+D_{Y_n}) - (f_n^*L_Z) \big) \in \overline{\NE}^1(Y_n).$
\end{claim}

\medskip

Let us for the moment assume that Claims~\ref{claim:1} and~\ref{claim:2} hold and proceed 
with the proof of Theorem~\ref{thm:vanishing}.
Let $f_n': Y_{n}' \to Z'$ be a flattening of $f_n$ (see the discussion preceding Lemma~\ref{lem:flat}),
with induced biratinal maps 
$\tau: Z'\to Z$ and $\sigma : Y_n ' \to Y_n$ leading to the following commutative diagram.

$$
\xymatrix{
Y_n'   \ar[rr]^{\sigma}  \ar[d]_{f'_n} && Y_n  \ar[d]^{f_n} \\
Z' \ar[rr]^{\tau}   &&    Z.
}
$$

Thanks to the solution of the relative log-abundance problem in $\dim(X)=3$
by Keel, Matsuki and McKernan~\cite{MR2057020}, using \ref{item:5k}, we have $(K_{Y_n} + D_{Y_n})|_{F_z}\equiv 0$, where $F_z$ is the general 
fiber of $f_n$. Therefore, by Lemma~\ref{lem:flat}, there exists a 
$\bQ$-divisor $A_{Z'}$ in $Z'$ such that 
$\sigma^*(K_{Y_n} + D_{Y_n}) \sim_{\bQ}   (f_n')^*(A_{Z'})$
so that 

$$
\sigma^*\big(   (K_{Y_n}  + D_{Y_n}) \pm f_n^*L_Z    \big)  \sim_{\bQ}   (f_n')^* (A_{Z'} \pm \tau^*L_Z).
$$
By Claim~\ref{claim:1} it thus follows that $\kappa(Z', A_{Z'} + \tau^*(L_Z)) = 
\kappa(Y_n, K_{Y_n}+ D_{Y_n}+ f_n^*L_Z)  = \dim(Z)$. 
Moreover, by using Claim~\ref{claim:2}, we find $(A_{Z'} - \tau^*(L_Z)) \in \overline{\NE}^1(Z')$. 
Therefore $A_{Z'}$ is big in $Z'$ and we have

\begin{equation}\label{eq:ineq1}
\kappa(Z' , A_{Z'})  \geq  \kappa(Z', \tau^*(L_Z))  =  \kappa(Z, L_Z).
\end{equation}
On the other hand, by using the negativity lemma, we have 

\begin{equation}\label{eq:ineq2}
\kappa(Z', A_{Z'})  = \kappa(Y_n , K_{Y_n} + D_{Y_n})  = \kappa(Y, K_Y+D_Y).
\end{equation}
By combining (\ref{eq:ineq1}) and (\ref{eq:ineq2}) we reach the inequality 

$$
\kappa(Y, L_Y)  \leq \kappa(Y, K_Y+D_Y).
$$

We now turn to proving Claims~\ref{claim:1} and \ref{claim:2}.

\medskip

\noindent
\emph{Proof of Claim~\ref{claim:1}.}
Let $(\widetilde Y, \widetilde D_Y)$ be a common log-smooth
higher birational model for $(Y,D_Y)$ and $(Y_n, D_{Y_n})$, with birational morphisms 
$\mu: \widetilde Y \to Y$ and $\mu_n :\widetilde Y\to Y_n$. According to~\cite[Lem.~3.38]{KM98}
there is an effective $\mu$-exceptional divisor $E_{\mu}$ and an effective $\mu_n$-exceptional 
divisor $E_{\mu_n}$ such that 

$$
\mu^*(K_Y+D_Y)  + E_{\mu}   \sim_{\bQ}    (\mu_n)^*(K_{Y_n}+ D_{Y_n})  + E_{\mu_n}.
$$
Therefore, we have 
$$
\mu^*\big( (K_Y+D_Y) + f^*L_Z \big) + E_{\mu}  \sim_{\bQ}  (\mu_n)^*(K_{Y_n}  + D_{Y_n} + f_n^*L_Z) + E_{\mu_n},
$$
which establishes the claim.

\medskip

\noindent
\emph{Proof of Claim~\ref{claim:2}.}
The proof will be based on induction on $n$. Assume that 
$$
\big( (K_{Y_{n-1}} + D_{Y_{n-1}}) - f_{n-1}^*(L_Z) \big) \in \overline{\NE}^1(Y_{n-1}).
$$
Now, the map $g_n: Y_{n-1}\dashrightarrow Y_n$ is either a divisorial contraction 
or a flip over $Z$. In the case of the latter the claim is easy to check, so let 
us assume that $g_n$ is a divisorial contraction and let $E_1$ and $E_2$ be 
two effective exceptional divisors such that the equivalence 

\begin{equation}\label{eq:sim}
K_{Y_{n-1}} + D_{Y_{n-1}}  + E_1  \sim_{\bQ}    g_n^*(K_{Y_n} + D_{Y_n})  + E_2
\end{equation}
holds. After subtracting $f_{n-1}^*(L_Z) = g_n^*(f_n^*L_Z)$ from both sides of (\ref{eq:sim}) and by using 
the inductive hypothesis 
we find that 

$$
\big(g_n^* ( K_{Y_n}  + D_{Y_n}  - f_n^*L_Z) + E_2 \big)  \in \overline{\NE}^1(Y_{n-1}),
$$
which implies $\big( K_{Y_n} + D_{Y_n} - f_n^*(L_Z)  \big)\in \overline{\NE}^1(Y_n)$.
This finishes the proof of Claim~\ref{claim:2}. \qed

\begin{remark}
From the proof of Theorem~\ref{thm:vanishing} it follows that if Log-Abundance 
Conjecture (for log-smooth pairs) holds in $\dim\leq n-2$, then 
Theorem~\ref{thm:vanishing} can be verified in dimension $n$.
\end{remark}

\begin{remark}
Theorem~\ref{thm:vanishing} is naturally related to a conjecture of Campana and Peternell,
where the authors predict that over a smooth projective variety $X$, if $K_X\sim_{\bQ} L + B$, 
where $L$ is effective and $B$ is pseudo-effective, then $\kappa(X)\geq \kappa(L)$, cf.~\cite{CP11}.
They also provide a proof to this conjecture when $B$ is numerically trivial \cite[Thm.~0.3]{CP11}.
(See also \cite{CKP12} for the generalization to the logarithmic setting.)
\end{remark}

\section{Concluding remarks and further questions}
\label{sect:Section5-Future}

We recall that a smooth quasi-projective variety $V$ of dimension $n$ is said to be 
\emph{special} if, for every invertible subsheaf $\sL\subseteq \Omega^p_X(\log D)$, the inequality 
$\kappa(X, \sL)< p$ holds, for all $1\leq p\leq n$. 
Here, by $(X,D)$ we denote a log-smooth compactification of $V$.
We refer to \cite{Cam04} for an 
in-depth discussion of this notion. We note that while varieties of Kodaira dimension 
zero form an important class of special varieties~\cite[Thm~5.1]{Cam04}, 
there are special varieties of every possible (but not maximal) Kodaira dimension. 

A conjecture of Campana predicts that a smooth projective family of canonically 
polarized manifolds $f_U : U \to V$ parametrized by a special quasi-projective 
variety $V$ is isotrivial. One can naturally extend this conjecture to the following 
setting. 

\begin{conjecture}\label{conj:iso}
Let $f_U: U \to V$ be a smooth projective family, where $V$ is equipped with 
a morphism $\mu: V\to P_h$ to the coarse moduli scheme $P_h$ associated with 
the moduli functor $\sP_h(\cdot)$ of polarized projective manifolds with 
semi-ample canonical bundle and fixed Hilbert polynomial $h$. If $V$ is special, then $f_V$ is isotrivial.   
\end{conjecture}

By using the refinement of \cite{VZ02}, due to Jabbusch and Kebekus \cite{MR2976311}, and 
the main result of \cite{CP13}, in the canonically polarized case, 
Conjecture~\ref{conj:iso} was established in \cite{taji16}.
We invite the reader to also consult \cite{CKT16}, Claudon \cite{Claudon15}, \cite{CP16}
and Schnell~\cite{Sch17}.

After a close inspection one can observe that the strategy of 
\cite{taji16} can be extended to establish \ref{conj:iso} in its full generality. 
More precisely, existence of the functor $\sP_h$ with an associated algebraic coarse 
moduli scheme gives us a twofold advantage. It ensures the 
existence of ``effective" Viehweg-Zuo subsheaves $\sL \subseteq \bigl(\Omega^1_X( \log D)\bigr)^{\otimes i}$ 
whose Kodaira dimension verifies $\kappa(\sL) \geq \Var(f_U)$, cf.~\cite{VZ02} and is
constructed at the level of moduli stacks, cf.~\cite{MR2976311}.
After imposing some additional orbifold structures
naturally arising from the induced moduli map $\mu: V \to P_h$, these two properties allow us 
to essentially reduce the problem to the case where $\sL$ is big, cf.~\cite[Thm.~4.3]{taji16}. 

\begin{theorem}\label{thm:iso}
Conjecture~\ref{conj:iso} holds.
\end{theorem}

Campana has kindly informed the author that, 
extending their current result \cite{AC17},
in a joint work with Amerik,
they have established Theorem~\ref{thm:iso} in a more general context of
 projective families with orbifold base.

Using the results of \cite{Cam04}, it is not difficult to trace a connection between 
the two Conjectures~\ref{conj:iso} and~\ref{conj:kk}; a solution to Conjecture~\ref{conj:iso}
 leads to a solution for Conjecture~\ref{conj:kk}. We record this observation in the 
 following theorem.

\begin{theorem}\label{thm:IsoToKK}
For any quasi-projective variety $V$ equipped with $\mu_V: V\to P_h$, induced by a family 
$f_U : U \to V$ of polarized manifolds, we have 

\begin{equation}\label{eq:end}
 \Var(f_U) \leq \kappa(X,D),
\end{equation}
$(X,D)$ being a log-smooth compactification of $V$.
\end{theorem}

Before proceeding to the proof of Theorem~\ref{thm:IsoToKK}, let us recall the notion 
of the \emph{core map} defined by Campana. Given a smooth 
quasi-projective variety $V$ that is not of log-general type, the core map
is a rational map  $c_X: X\dashrightarrow Z$ satisfying the following two key properties. 

\begin{enumerate}
\item \label{item:core1} $c_X$ is almost holomorphic with special fibers.
\item \label{item:core2} $c_X$ is birationally equivalent to a 
fiber space $c_{\widetilde X}: (\widetilde X, \widetilde D)
\to (\widetilde Z,  \Delta_{\widetilde Z})$, where $\Delta_{\widetilde Z}\in \Div_{\bQ}(\widetilde Z)$ and 
the pair $(\widetilde Z,  \Delta_{\widetilde Z})$
is a log-smooth orbifold base for $c_{\widetilde X}$ and is of log-general type. 
\end{enumerate}

\begin{proof}
Assume that $f_U$ is not isotrivial and $V$ is not of log-general type. 
Then, by Item~\ref{item:core1} and Theorem~\ref{thm:iso}, 
the compactification $\overline{\mu}_V: X\to \overline{P_h}$ 
of $\mu_V$ factors through 
the core map $c_X: X \dashrightarrow Z$ with positive dimensional fibers. 
In particular we have
\begin{equation}\label{eq:ineqVar}
\Var(f_U) \leq \dim(Z).
\end{equation}
The theorem then follows from Campana's
orbifold $C_{n,m}^{\rm{orb}}$ theorem for any orbifold, log-general type fibration;
an example of which is $c_{\widetilde X}$.
More precisely, by using \ref{item:core2} and \cite[Thm.~4.2]{Cam04} we can conclude that the inequality

$$
\kappa(\widetilde X, \widetilde D) \geq \kappa(\widetilde Z, \Delta_{\widetilde Z})
$$
holds, which, together with (\ref{eq:ineqVar}) and the inequality $\kappa(X, D) \geq \kappa(\widetilde X, \widetilde D)$,
establishes the theorem. 
\end{proof}

\medskip

When the fibers are of general type, it is conceivable that one may be able to use the result of Birkar, Cascini, Hacon
and McKernan
\cite{BCHM10} on the existence of good minimal models for varieties of general type
(and relative base point freeness theorem) to reduce to the case of maximal variation.
More precisely, the arguments of \cite[Lem.~2.8]{VZ02}
combined with Hodge theoretic construction of \cite{PS15}, or the ones
in Section~\ref{sect:Section3-VZ} of the current paper, may allow for the construction of
a Viehweg-Zuo subsheaf
$\sL$ over a new base where variation of the pulled back family is maximal.
If so, in this case the discussion 
prior to Theorem \ref{thm:iso} again applies and Theorem~\ref{thm:iso} and consequently Theorem~\ref{thm:IsoToKK} 
hold.
But as pointed out from the outset, 
the main remaining difficulty is solving the isotriviality problem
in the absence of a well-behaved functor, such as $\sP_h$, 
detecting the variation in the family (after running the minimal model program, if necessary).
At the moment, it is not clear to the author 
that one can expect Conjecture~\ref{conj:iso} to hold for all
smooth projective families whose fibers have semi-ample canonical bundle or---even more generally---admit
a good minimal model. 

\begin{question}\label{Q}
Let $f_U: U\to V$ be a smooth projective family of manifolds admitting 
a good minimal model. If $V$ is special, (apart from the cases discussed above) 
is it true that $f_U$ is isotrivial?
\end{question}

\begin{remark}
Question~\ref{Q}, which is a generalization of Campana's Isotriviality Conjecture,
has been recently affirmatively answered by the author, cf.~\cite{taji20}.
\end{remark}

\bibliography{bibliography/general}


\end{document}